\documentclass{amsart}
              [1997/02/02 v1.2e PROC Author Class]

\copyrightinfo{2006}
  {American Mathematical Society}

\newtheorem{theorem}{Theorem}[section]
\newtheorem{corollary}[theorem]{Corollary}

\newtheorem{proposition}[theorem]{Proposition}
\newtheorem{example}[theorem]{Example}
\newtheorem{question}[theorem]{Question}
\newtheorem{remark}[theorem]{Remark}

\Large

\begin{document}

\title[Continuous extension of functions from countable sets]{Continuous extension of functions from countable sets}

\author{V.Mykhaylyuk}
\address{Department of Mathematics\\
Chernivtsi National University\\ str. Kotsjubyn'skogo 2,
Chernivtsi, 58012 Ukraine}
\email{vmykhaylyuk@ukr.net}

\subjclass[2000]{Primary 54C20, 54C35, Secondary 46E10, 54C05, 54C45}


\commby{Ronald A. Fintushel}


\keywords{extension property, linear extender, $C$-embedding, retract, stratifiable space, $P$-point, Stone-$\breve{C}$ech compactification}

\begin{abstract}
   We give a characterization of countable discrete subspace $A$ of a topological space $X$ such that there exists a (linear) continuous mapping $\varphi:C_p^*(A)\to C_p(X)$ with $\varphi(y)|_A=y$ for every $y\in C_p^*(A)$. Using this characterization we answer two questions of A.~Arhangel'skii. Moreover, we introduce the notion of well-covered subset of a topological space and prove that for well-covered functionally closed subset $A$ of a topological space $X$ there exists a linear continuous mapping $\varphi:C_p(A)\to C_p(X)$ with $\varphi(y)|_A=y$ for every $y\in C_p(A)$.

\end{abstract}

\maketitle
\section{Introduction}

For a topological space $X$ we denote the space of all continuous function $y:X\to\mathbb R$ with the topology of pointwise convergence by $C_p(X)$, and the subspace of all continuous bounded function $y:X\to\mathbb R$ is denoted by $C^*_p(X)$.

According to the well-known Tietze-Urysohn theorem, for a normal space $X$ and a closed subset $A$ of $X$ there exists a mapping $\varphi: C(A)\to C(X)$ such that $\varphi(y)|_A=y$ for every $y\in C(A)$. The existence and properties (linearity, continuity with respect different topologies, etc.) of such extender $\varphi: C(A)\to C(X)$ for various classes of spaces $X$ were investigated by many mathmeticians (see, for instance, \cite{Du}, \cite{Bor}, \cite{Sen}, \cite{St}, \cite{SV}, \cite{GHO}, \cite{Yam} and literature given there). In particular, the existence of a linear continuous extender $\varphi: C_p(A)\to C_p(X)$ for every closed subset $A$ of a stratifiable space $X$ was obtained in \cite[Theorem 4.3]{Bor} and  the existence of such extender for every closed subset $A$ of locally compact generalized ordered space $X$ was proved in \cite[Corollary 1]{GHO}).

The following questions were published in \cite{Arh} (see also Questions 4.10.3 and 4.10.11 from \cite{Tkachuk}).

\begin{question}\label{q:1.1}
Let $X$ be a pseudocompact space such that for any countable set $A\subseteq X$ there exists a (linear) continuous map $\varphi:C_p^*(A)\to C_p(X)$ with $\varphi(y)|_A=y$ for every $y\in C_p(A)$. Must $X$ be finite?
\end{question}

\begin{question}\label{q:1.2}
Let $X$ be the subspace of all weak $P$-points of $\beta\omega\setminus\omega$. It is true that for any countable set $A\subseteq X$ there exists a (linear) continuous map $\varphi:C_p^*(A)\to C_p(X)$ such that $\varphi(y)|_A=y$ for every $y\in C_p(A)$?
\end{question}

A point $x$ of a topological space $X$ is called {\it an weak $P$-point} if $x\not \in \overline{A}$ for every countable set $A\subseteq X\setminus \{x\}$.

In this paper we give a characterization of countable discrete subspace $A$ of a topological space $X$ for which exists a (linear) continuous mapping $\varphi:C_p^*(A)\to C_p(X)$ with $\varphi(y)|_A=y$ for every $y\in C_p^*(A)$. Using this characterization we obtain the positive answer to Question \ref{q:1.1} and the negative answer to Question \ref{q:1.2}. Moreover, we introduce the notion of well-covered subset of a topological space and prove that for well-covered functionally closed subset $A$ of a topological space $X$ there exists a linear continuous mapping $\varphi:C_p^*(A)\to C_p(X)$ with $\varphi(y)|_A=y$ for every $y\in C_p^*(A)$.

\section{Countable $C$-embedding sets}

The next property is probably well-known (see \cite[Corollary 1]{Sen}).

\begin{proposition}\label{p:3.1} Let $X$ be a completely regular space and $A\subseteq X$ such that there exists a continuous mapping $\varphi:C_p^*(A)\to C_p(X)$ such that $\varphi(y)|_A=y$ for every $y\in C_p^*(A)$. Then the set $A$ is closed in $X$. \end{proposition}

\begin{proof} Suppose that $x_0\in \overline{A}\setminus A$. Let $y_0(x)=0$ for every $x\in A$ and $z_0=\varphi(y_0)$. Clearly that $z_0(x_0)=0$. Consider the neighborhood $W_0=\{z\in C_p(X):z(x_0)<\tfrac12\}$ of $z_0$ in $C_p(X)$ and choose an finite set $B\subseteq A$ such that $\varphi(y)\in W_0$ for every $y\in C_p^*(A)$ with $y|_B=y_0|_B$. We choose an open in $X$ set $G$ such that $B\subseteq G\cap A$ and $x_0\not \in \overline{G}$. There exists a continuous function $y\in C_p^*(A)$ such that $y(x)=0$ for every $x\in B$ and $y(x)=1$ for every $x\in A\setminus G$. It easy to see that $y|_B=y_0|_B$ and $\varphi(y)\not\in W_0$, which implies a contradiction.
\end{proof}

A set $A$ in a topological space $X$ is called {\it strongly functionally discrete in $X$} if there exists a discrete family $(G_a;a\in A)$ of functionally open sets $G_a\ni a$.

\begin{proposition}\label{p:2.1} Let $X$ be a topological space, $A=\{a_n:n\in\mathbb N\}\subseteq X$ be a countable set, $Y$ be a compact and $f:X\times Y\to\mathbb R$ be a separately continuous function such that the continuous mapping $\varphi:Y\to C_p(A)$, $\varphi(y)(a)=f(a,y)$, is a homeomorphic embedding and for every $n\in\mathbb N$ there exist $y'_n, y_n''\in Y$ with $|f(a_k,y_n')-f(a_k,y_n'')|\leq \tfrac{1}{n+1}$ for all $k=1,\dots , n$ and $|f(a_k,y_n')-f(a_k,y_n'')|\geq 1$ for all $k> n$. Then the set $A$ is a strongly functionally discrete in $X$.
\end{proposition}

\begin{proof} Without loss of generality we can supose that $Y\subseteq C_p(A)$. For every $n\in\mathbb N$ we put
$$
U_n=\bigcap_{k=1}^{n-1}\{x\in X:|f(x,y_k')-f(x,y_k'')|>\tfrac56\}\bigcap\{x\in X:|f(x,y_n')-f(x,y_n'')|<\tfrac23 \}.
$$
Clearly, $a_n\in U_n$ and $U_n$ is functionally open in $X$ for every $n\in\mathbb N$.

We show that $(U_n)_{n=1}^{\infty}$ is discrete in $X$. Fix $x_0\in X$. Since $Y\subseteq C_p(A)$ is a Hausdorff compact space and the mapping $f^{x_0}:Y\to\mathbb R$ is continuous, there exists $n_0\in\mathbb N$ such that $|f(x_0,y')-f(x_0,y'')|<\tfrac12$ for every $y',y''\in Y$ with $|y'(a_k)-y''(a_k)|<\tfrac{1}{n_0}$ for every $k=1,\dots n_0$. Then for the neighborhood
$$
U_0=\{x\in X:|f(x,y_{n_0}')-f(x,y_{n_0}'')<\tfrac12|\}
$$
of $x_0$ in $X$ we have $U_0\cap U_n=\emptyset$ for every $n>n_0$. Thus, $(U_n)_{n=1}^{\infty}$ is locally finite in $X$. It remains to use that the sequence $(F_n)_{n=1}^\infty$ of the sets
$$
F_n=\bigcap_{k=1}^{n-1}\{x\in X:|f(x,y_k')-f(x,y_k'')|\geq\tfrac56\}\bigcap\{x\in X:|f(x,y_n')-f(x,y_n'')|\leq\tfrac23 \}
$$
is disjoint.
\end{proof}

\begin{theorem}\label{th:2.2} Let $X$ be a topological space and $A\subseteq X$ be a discrete countable subspace of $X$. Then the following conditions are equivalent

$(i)$\,\, there exists a (linear) continuous mapping $\varphi:C_p^*(A)\to C_p^*(X)$ such that $\varphi(y)|_A=y$ for every $y\in C_p^*(A)$;

$(ii)$\,\, there exists a (linear) continuous mapping $\varphi:C_p(A)\to C_p(X)$ such that $\varphi(y)|_A=y$ for every $y\in C_p(A)$;

$(iii)$\,\, there exists a (linear) continuous mapping $\varphi:C_p^*(A)\to C_p(X)$ such that $\varphi(y)|_A=y$ for every $y\in C_p^*(A)$;

$(iv)$\,\,the set $A$ is strongly functionally discrete in $X$.
\end{theorem}

\begin{proof} The implications $(i)\Rightarrow(iii)$ and $(ii)\Rightarrow(iii)$ are obvious.

$(iii)\Rightarrow(iv)$. Let $Y=C_p(A,\{0,1\})$ and let $f:X\times Y\to\mathbb R$ is defined by $$f(x,y)=\varphi(y)(x).$$
Clearly, $f$ is separately continuous and the set $A$ is strongly functionally discrete in $X$ according to Proposition \ref{p:2.1}.

$(iv)\Rightarrow(i)$ and $(iv)\Rightarrow(ii)$. Let $(G_a;a\in A)$ is a discrete family of functionally open sets $G_a\ni a$. For every $a\in A$ we choose a continuous mapping  $\varphi_a:X\to [0,1]$ with $\varphi_a^{-1}((0,1])=G_a$ and $\varphi_a(a)=1$. Notice that the mapping $\varphi:C_p(A)\to C_p(X)$,
$$\varphi(y)(x)=\sum_{a\in A}\varphi_a(x)y(a),$$
and its restriction $\varphi_{C^*_p(A)}$ are the required one.
\end{proof}

The following corollary gives the answers to Questions \ref{q:1.1} and \ref{q:1.2}.

\begin{corollary}\label{cor:2.3} Let $X$ be a topological space such that for every countable set $A\subseteq X$ there exists a continuous mapping $\varphi:C_p^*(A)\to C_p(X)$ with $\varphi(y)|_A=y$ for every $y\in C_p^*(A)$. Then

$(a)$\,\,every discrete countable subspace $A$ of $X$ is a strongly functionally discrete set in $X$;

$(b)$\,\,if $X$ is a $T_2$-space in which every locally finite system of functionally open sets is finite (in particular, if $X$ is a pseudocompact) then $X$ is finite;

$(c)$\,\,$X$ does not equal to the space of all weak $P$-points in $\beta\omega\setminus \omega$;

$(d)$\,\,if $X$ is a completely regular space, then every countable subspace $A$ of $X$ is a strongly functionally discrete set in $X$.
\end{corollary}

\begin{proof} The statement $(a)$ follows immediately from Theorem \ref{th:2.2}.

$(b)$. It is enough to note that any infinite $T_2$-space has a countable discrete subspace.

$(c)$. According to \cite{GO} the space $W$ of all weak $P$-points in $\beta\omega\setminus \omega$ is ultrapseudocompact, in particular, pseudocompact. Moreover, pseudocompactness of $W$ easy follows from the next fact (see \cite{K}, \cite{vM}): {\it there exists a weak $P$-point $z$ in $\beta\omega\setminus \omega$ which is not a $P$-point in $\beta\omega\setminus \omega$} (a point $x$ of a topological space $X$ is called {\it a $P$-point} if $x\not \in \overline{A}$ for every $F_\sigma$-set $A\subseteq X\setminus \{x\}$).

$(d)$. According to Proposition \ref{p:3.1}, every countable subset $A\subseteq X$ is closed. It easy to see that every countable subset $A\subseteq X$ is discrete. Now it remains to use $(a)$.
\end{proof}

\begin{remark}\label{r:3.4}
Since the closure $\overline{A}$ in $\beta\omega$ of a countable set $A\subseteq \beta\omega$ is homeomorphic to $\beta A$ (see, for example, \cite[Corollary 3.2]{My}), every countable set $A\subseteq \beta\omega$ is $C^*$-embedding in $\beta\omega$. But for every infinite pseudocompact space $X\subseteq \beta\omega$ and every countable set $A\subseteq X$ does not exist a continuous mapping $\varphi:C_p^*(A)\to C_p(X)$ with $\varphi(y)|_A=y$ for every $y\in C_p^*(A)$ as follows from Corollary \ref{cor:2.3}(b).
\end{remark}

A space $X$ is called {\it a $P$-space} if every $x\in X$ is a $P$-point.

\begin{remark}\label{r:3.5}
It easy to see that every countable subset $A$ of a completely regular $P$-space $X$ is a strongly functionally discrete set in $X$. Hence, for every countable subset $A$ there exists a linear continuous mapping $\varphi:C_p^*(A)\to C_p(X)$ such that $\varphi(y)|_A=y$ for every $y\in C_p^*(A)$ according to Theorem \ref{th:2.2}.
\end{remark}

The following example shows that there exists a space $X$ with this property which is not a $P$-space.

\begin{example}
Let $S$ be an uncountable set and $X$ is the space of all function $x:S\to\{0,1\}$, i.e. $X=\{0,1\}^S$. Let $x_0(s)=0$ for every $s\in S$. The set $X\setminus \{x_0\}$ is open subset of $X$ equipped with the topology of uniform convergence on the countable subsets $T\subseteq S$. Moreover, the sets
$$
U(T,B)=(\{0,1\}^T\setminus B)\times \{0,1\}^{S\setminus T},
$$
where $T\subseteq S$ is a countable set and $B\subseteq \{0,1\}^T$ is a countable set with $x_0|_{T}\not\in B$, form a base of neighborhoods of $x_0$ in $X$.
It is easy to see that every countable set $A\subseteq X$ is strongly functionally discrete in $X$.

On other hand, let $\{t_n:n\in \omega\}\subseteq S$ be a countable set. Then the set
$$
G=\bigcap_{n\in \omega}\{0\}\times \{0,1\}^{S\setminus \{t_n\}}
$$
is a $G_\delta$-set in $X$ with $x_0\in G$, but $G$ is not a neighborhood of $x_0$. Thus, $x_0$ is not a $P$-point in $X$.
\end{example}

\section{Continuous extension from closed sets}

\begin{proposition}\label{p:4.2} Let $X$ be a topological space, $A\subseteq X$, $F\subseteq X$ be a functionally closed set and $G\subseteq X$ be a functionally open set such that $A\subseteq F\subseteq G$ and $A$ is a retract of $G$. Then there exists a linear continuous mapping $\varphi:C_p^*(A)\to C_p^*(X)$ such that $\varphi(y)|_A=y$ for every $y\in C_p^*(A)$.
\end{proposition}

\begin{proof} Choose a retraction $r:G\to A$ and a continuous function $f:X\to [0,1]$ such that $F\subseteq f^{-1}(1)$ and $X\setminus G\subseteq f^{-1}(0)$. It remains to put $\varphi(y)(x)=f(x)\cdot y(r(x))$ for every $x\in X$.
\end{proof}

Let $X$ be a topological space and $\pi_X:X\to C_p(C_p(X))$, $\pi_X(x)(y)=y(x)$. We denote the linear subspace
$$
L(X)=\{\alpha_1\pi_X(x_1)+\cdots +\alpha_n\pi_X(x_n):n\in\omega,\,x_1,\dots ,x_n\in X,\, \alpha_1,\dots, \alpha_n\in\mathbb R\}.
$$
of $C_p(C_p(X))$ by $L(X)$. Analogously we put $\pi^*_X:X\to C_p(C^*_p(X))$, $\pi^*_X(x)(y)=y(x)$ and denote the linear subspace
$$
L^*(X)=\{\alpha_1\pi_X^*(x_1)+\cdots +\alpha_n\pi_X^*(x_n):n\in\omega,\,x_1,\dots ,x_n\in X,\, \alpha_1,\dots, \alpha_n\in\mathbb R\}.
$$
of $C_p(C^*_p(X))$ by $L^*(X)$.

\begin{proposition}\label{p:4.1} Let $X$ be a topological space and $A\subseteq X$. Consider the following conditions

$(i)$\,\,there exists a linear continuous mapping $\varphi:C_p(A)\to C_p(X)$ with $\varphi(y)|_A=y$ for every $y\in C_p(A)$;

$(ii)$\,\,there exists a continuous mapping $\psi:X\to L(A)$ with $\psi(a)=\pi_A(a)$ for every $a\in A$;

$(iii)$\,\,there exists a linear continuous mapping $\varphi:C_p^*(A)\to C_p(X)$ with $\varphi(y)|_A=y$ for every $y\in C_p^*(A)$;

$(iv)$\,\,there exists a continuous mapping $\psi:X\to L^*(A)$ with $\psi(a)=\pi^*_A(a)$ for every $a\in A$.

Then $(i)\Leftrightarrow (ii)\Rightarrow (iii)\Leftrightarrow (iv)$. If $A$ is homeomorphic to a topological vector space, then $(iii)\Rightarrow (ii)$ and all conditions $(i)-(iv)$ are equivalent to the following condition

$(v)$\,\,$A$ is a retract of $X$.
\end{proposition}

\begin{proof} $(i)\Rightarrow(ii)$. We consider the continuous mapping $\psi:X\to C_p(C_p(A))$, $\psi(x)(y)=\varphi(y)(x)$. It remains to note that according to \cite[Theorem IV.1.2]{Sch}, for every $x\in X$ there exists $z\in L(A)$ such that $\psi(x)=z$.

$(ii)\Rightarrow(i)$. It is sufficient to put $\varphi(y)(x)=\psi(x)(y)$ for every $x\in X$ and $y\in C_p(A)$.

The implication $(iii)\Leftrightarrow(iv)$ can be proved similarly.

Let $A$ is homeomorphic to a topological vector space. Without loss of the generality we can assume that $A$ is a topological vector space. Then the implication $(iv)\Rightarrow(ii)$ follows immediately from the fact that the mapping $\phi:L(A)\to L^*(A)$, $\phi(z)=z|_{C_p^*(A)}$, is a homeomorphism. Moreover, the mapping $r:X\to A$,
$$
 r(x)=\left\{\begin{array}{ll}
                         a, & x\in X\setminus A\,\,{\rm and}\,\,\psi(x)=\pi_A(a)\\
                         x, & x\in A
                       \end{array}
 \right.
$$
is continuous.
\end{proof}

We say that a subset $A$ of a topological space $X$ is {\it a $L$-retract in $X$}, if there exists a continuous mapping $\psi:X\to L(A)$ with $\psi(a)=\pi_A(a)$ for every $a\in A$.

A subset $A$ of a topological space $X$ is called {\it well-covered in $X$}, if there exists a sequence of locally finite functionally open covers $(U(n,i):i\in I_n)$ of $A$ in $X$ and sequence of families $(\lambda_{n,i}:i\in I_n)$ of continuous mappings $\lambda_{n,i}:\overline{U(n,i)}\to A$ such that for every $a\in A$ and every neighborhood $U$ of $a$ in $A$ there exist a $n_0\in \omega$ and a neighborhood $U_0$ of $a$ in $X$ such that $\lambda_{n,i}(U(n,i)\cap U_0)\subseteq U$ for every $n\geq n_0$.

\begin{theorem}\label{th:4.3}
Let $X$ be a topological space and $A$ be an well-covered functionally closed subset of $X$. Then $A$ is  a $L$-retract in $X$.
\end{theorem}

\begin{proof} We choose a sequence of locally finite functionally open covers $(U(n,i):i\in I_n)$ of $A$ in $X$ and sequence of continuous mappings $(\lambda_{n,i}:i\in I_n)$ of continuous mappings $\lambda_{n,i}:U(n,i)\to A$ which satisfy the condition from the definition of well-covered set. Note that every set $G_n=\bigcup\limits_{i\in I_n}U(n,i)$ is functionally open in $X$. For every $n\in \omega$ we choose a partition of the unit $(\phi_{n,i}:i\in I_n)$ on $G_n$ which is subordinated to $(U(n,i):i\in I_n)$. For a fixed $n\in \omega$ and $i\in I_n$ we consider the mapping $\mu_{n,i}:X\to L(A)$,
$$
\mu_{n,i}(x)=\left\{\begin{array}{ll}
                         \phi_{n,i}(x)\pi_A(\lambda_{n,i}(x)), & x\in U(n,i)\\
                         0, & x\in X\setminus U(n,i).
                       \end{array}
 \right.
$$
Note that $\mu_{n,i}(x)= \phi_{n,i}(x)\pi_A(\lambda_{n,i}(x))$ for every $x\in \overline{U(n,i)}$. Therefore the restrictions of $\mu_{n,i}$ on the closed sets $\overline{U(n,i)}$ is continuous. Thus $\mu_{n,i}$ is continuous too.

 Now we choose sequences of functionally open in $X$ sets $W_n$ and functionally closed in $X$ sets $F_n$ such that $A\subseteq F_{n+1}\subseteq W_n\subseteq F_n\subseteq G_n$ for every $n\in \omega$, and a sequence of continuous functions $\delta_n:X\to[0,1]$ such that $\delta_n(X\setminus W_n)=\{0\}$ and $\delta_n(F_{n+1})=\{1\}$. Let the mapping $\psi_0:X\to L(A)$ is defined by $\psi_0(x)=0$. Moreover, for every $n\in \omega$ the mapping $\psi_n:G_n\to L(A)$ is defined by $\psi_n(x)=\sum\limits_{i\in I_n}\mu_{n,i}(x)$. Obviously, all mappings $\psi_n$ are continuous.
Now we consider the mapping $\psi:X\to L(A)$,
$$
\psi(x)=\left\{\begin{array}{lll}
                         \psi_0(x), & x\in X\setminus W_1\\
                         (1-\delta_n(x))\psi_{n-1}(x)+ \delta_n(x)\psi_n(x), & n\in \omega,\,\,x\in W_n\setminus W_{n+1}\\
                         \pi_X(x), & x\in A.
                       \end{array}
 \right.
$$
It is clear that $\psi$ is continuous at every point $x\in X\setminus A=(X\setminus W_1)\cup\bigcup\limits_{n=1}^\infty(W_n\setminus W_{n+1})$.

Fix a point $x_0\in A$ and show that $\psi$ is continuous at $x_0$. It is sufficient to prove that for every $y\in C_p(X)$ there exists a neighborhood $U$ of $x_0$ in $X$ such that $|\psi(x)(y)-\psi(x_0)(y)|<1$ for each $x\in U$.

Let $y_0\in C_p(X)$ and $U$ be a neighborhood of $x_0$ in $X$ such that $|y_0(x)-y_0(x_0)|<1$ for every $x\in U$. According to the choice of $U(n,i)$ and $\lambda_{n,i}$, there exist a $n_0\in \omega$ and a neighborhood $U_0$ of $x_0$ such that $\lambda_{n,i}(U(n,i)\cap U_0)\subseteq U$ for every $n\geq n_0$. Then for every $n\geq n_0$ and $x\in U_0\cap G_n$ we have
$$
|\psi_n(x)(y_0)-\psi(x_0)(y_0)|=|\sum\limits_{i\in I_n}\phi_{n,i}(x)\pi_A(\lambda_{n,i}(x))(y_0)-y_0(x_0)|=
$$
$$
=|\sum\limits_{i\in I_n}\phi_{n,i}(x)y_0(\lambda_{n,i}(x))-y_0(x_0)|\leq \sum\limits_{i\in I_n}\phi_{n,i}(x)|y_0(\lambda_{n,i}(x))-y_0(x_0)|<
$$
$$
<\sum\limits_{i\in I_n}\phi_{n,i}(x)=1.
$$
Now it is easy to see that $|\psi(x)(y_0)-\psi(x_0)(y_0)|<1$ for each $x\in U_0\cap W_{n_0+1}$.
\end{proof}

The following proposition shows that Theorem \ref{th:4.3} is a generalization of Theorem 4.3 from \cite{Bor}.

\begin{proposition}\label{pr:4.4}
Let $X$ be a stratifiable space. Then every closed in $X$ set $A\subseteq X$ is well-covered in $X$.
\end{proposition}

\begin{proof} According to the definition of stratifiable space there exists a mapping $G$ which assigns to each $n\in \omega$ and a closed subset $H\subseteq X$, an open set $G(n,H)$ contained $H$ such that

$(1)$\,\, $H=\bigcap_{n\in \omega} \overline{G(n,H)}$;

$(2)$\,\, $G(n,H)\subseteq G(n,K)$ for every closed subsets $H\subseteq K\subseteq X$ and $n\in \omega$.

Note that without loss of the generality we may assume that $G(n+1,H)\subseteq G(n,H)$ for all $n\in \omega$ and all open sets $H\subseteq X$. Let $A$ be a closed subset of $X$. For every $n\in \omega$ we choose a locally finite in $X$ open refinement $(U(n,i):i\in I_n)$ of $(G(n,\{x\}):x\in A)$. For every $n\in \omega$ and $i\in I_n$ we choose $x(n,i)\in A$ with $U(n,i)\subseteq G(n,\{x(n,i)\})$ and put $\lambda_{n,i}(x)=x(n,i)$ for every $x\in \overline{U(n,i)}$. We show that sequences of covers $(U(n,i):i\in I_n)$ and families $(\lambda_{n,i}:i\in I_n)$ satisfy the condition from the definition of well-covered set.

Let $x_0\in A$ and $U$ be an open neighborhood of $x_0$. We choose $n_0\in \omega$ such that $x_0\not\in \overline{G(n_0, X\setminus U)}$ and put $U_0=X\setminus \overline{G(n_0, X\setminus U)}$. Let $n\geq n_0$ and $i\in I_n$ with $U(n,i)\cap U_0\ne \emptyset$. Since $U(n,i)\subseteq G(n,\{x(n,i)\})$, $G(n,\{x(n,i)\})\cap U_0\ne \emptyset$, i.e. $G(n,\{x(n,i)\})\not\subseteq G(n_0, X\setminus U)$. Taking into account that $G(n, X\setminus U)\subseteq G(n_0, X\setminus U)$ we obtain that $G(n,\{x(n,i)\})\not\subseteq G(n, X\setminus U)$. Thus, $x_{n,i}\in U$ according to $(2)$.
\end{proof}

\begin{remark}\label{r:4.5}
The Sorgenfrey line $\mathbb S$ is an example of a perfectly normal non-stratifiable space in which every closed subset is well-covered in $\mathbb S$.
\end{remark}

\bibliographystyle{amsplain}

\end{document}